\documentclass[10pt]{amsart}
\usepackage{amssymb,MnSymbol}
\usepackage{amsthm,amsmath}

\title[]{On the extensions of Barlow-Proschan importance index and system signature to dependent lifetimes} % and its analogy with system signature

\author{Jean-Luc Marichal}
\address{Mathematics Research Unit, FSTC, University of Luxembourg, 6, rue Coudenhove-Kalergi, L-1359 Luxembourg, Luxembourg}
\email{jean-luc.marichal[at]uni.lu }

\author{Pierre Mathonet}
\address{University of Li\`ege, Department of Mathematics, Grande Traverse, 12 - B37, B-4000 Li\`ege, Belgium}
\email{p.mathonet[at]ulg.ac.be }

\date{August 29, 2012}

\begin{document}

\theoremstyle{plain}
\newtheorem{theorem}{Theorem}
\newtheorem{lemma}[theorem]{Lemma}
\newtheorem{proposition}[theorem]{Proposition}
\newtheorem{corollary}[theorem]{Corollary}
\newtheorem{fact}[theorem]{Fact}
\newtheorem*{main}{Main Theorem}

\theoremstyle{definition}
\newtheorem{definition}[theorem]{Definition}
\newtheorem{example}[theorem]{Example}

\theoremstyle{remark}
\newtheorem*{conjecture}{onjecture}
\newtheorem{remark}{Remark}
\newtheorem{claim}{Claim}

\newcommand{\N}{\mathbb{N}}
\newcommand{\R}{\mathbb{R}}
\newcommand{\Q}{\mathbb{Q}}
\newcommand{\Vspace}{\vspace{2ex}}
\newcommand{\bfx}{\mathbf{x}}
\newcommand{\bfy}{\mathbf{y}}
\newcommand{\bfh}{\mathbf{h}}
\newcommand{\bfe}{\mathbf{e}}
\newcommand{\os}{\mathrm{os}}
\newcommand{\dd}{\,\mathrm{d}}

\begin{abstract}
For a coherent system the Barlow-Proschan importance index, defined when the component lifetimes are independent, measures the probability that
the failure of a given component causes the system to fail. Iyer (1992) extended this concept to the more general case when the component
lifetimes are jointly absolutely continuous but not necessarily independent. Assuming only that the joint distribution of component lifetimes has no ties, we give an explicit expression for this extended index in terms of
the discrete derivatives of the structure function and provide an interpretation of it as a probabilistic value, a concept introduced in game
theory. This enables us to interpret Iyer's formula in this more general setting. We also discuss the analogy between this concept and that of system signature and show how it can be used to define a symmetry index for
systems.
\end{abstract}



\maketitle

%---------------------------------------------------------------------------------------------- Section 1
\section{Introduction}

Consider an $n$-component system $S=(n,\phi,F)$, where $\phi$ denotes the associated structure function $\phi\colon\{0,1\}^n\to\{0,1\}$ (which
expresses the state of the system in terms of the states of its components) and $F$ denotes the joint c.d.f.\ of the component lifetimes
$X_1,\ldots,X_n$, that is,
$$
F(t_1,\ldots,t_n) ~=~ \Pr(X_1\leqslant t_1,\ldots,X_n\leqslant t_n),\qquad t_1,\ldots,t_n\geqslant 0.
$$
We assume that the system is \emph{semicoherent}, i.e., the
structure function $\phi$ is nondecreasing in each variable and satisfies the conditions $\phi(0,\ldots,0)=0$ and $\phi(1,\ldots,1)=1$.

To avoid cumbersome notation, we identify Boolean $n$-tuples $\bfx\in\{0,1\}^n$ and subsets $A\subseteq [n]=\{1,\ldots,n\}$ by setting $x_i=1$ if and only if $i\in A$. We thus use the same symbol to denote both a function $f\colon\{0,1\}^n\to\R$ and its corresponding set function $f\colon 2^{[n]}\to\R$, interchangeably. For instance we have $\phi(1,\ldots,1)=\phi([n])$.

An \emph{importance index} was introduced in 1975 by Barlow and Proschan \cite{BarPro75} for systems whose components have continuous and
independent lifetimes as the $n$-tuple $\mathbf{I}_{\mathrm{BP}}$ whose $j$th coordinate $I_{\mathrm{BP}}^{(j)}$ is the probability that the
failure of component $j$ causes the system to fail. In mathematical terms,
\begin{equation}\label{eq:sa9fd7a}
I_{\mathrm{BP}}^{(j)} ~=~ \Pr(T=X_{j}),\qquad j\in [n],
\end{equation}
where $T$ denotes the system lifetime.

When the components have i.i.d.\ lifetimes, this index reduces to the $n$-tuple $\mathbf{b}=(b_1,\ldots,b_n)$, where
\begin{equation}\label{eq:asd465d}
b_j ~=~ \sum_{A\subseteq [n]\setminus\{j\}}\frac{1}{n{n-1\choose |A|}}\,\Delta_j\phi(A)~=~\frac{1}{n}\,\sum_{k=0}^{n-1}\frac{1}{{n-1\choose
k}}\,\sum_{\textstyle{A\subseteq[n]\setminus\{j\}\atop |A|=k}}\Delta_j\phi(A)
\end{equation}
and $\Delta_j\phi(A)=\phi(A\cup\{j\})-\phi(A\setminus\{j\})$. Thus, in the i.i.d.\ case the
probability $(\ref{eq:sa9fd7a})$ does not depend of the c.d.f.\ $F$. Due to this feature it is sometimes referred to as a \emph{structural
importance}. Note that the expression on the right-hand side of (\ref{eq:asd465d}) was actually
defined in 1953 in cooperative game theory, where it is known as the Shapley-Shubik value \cite{Sha53,ShaShu54}.

The concept of \emph{signature}, which reveals a strong analogy with that of Barlow-Proschan importance index  (\ref{eq:sa9fd7a}), was introduced in 1985 by Samaniego \cite{{Sam85}} (see also \cite{Sam07}) for systems whose
components have continuous and i.i.d.\ lifetimes as the $n$-tuple $\mathbf{s}=(s_1,\ldots,s_n)$, where $s_k$ is the probability that the $k$th
component failure causes the system to fail. That is,
$$%\begin{equation}\label{eq:sa9fd7}
s_k ~=~ \Pr(T=X_{k:n}),
$$%\end{equation}
where $X_{k:n}$ denotes the $k$th smallest lifetime, i.e., the $k$th order statistic obtained by rearranging the variables $X_1,\ldots,X_n$ in
ascending order of magnitude.

Boland \cite{Bol01} showed that $s_k$ can be explicitly written in the form
\begin{equation}\label{eq:asad678}
s_k ~=~ \sum_{\textstyle{A\subseteq [n]\atop |A|=n-k+1}}\frac{1}{{n\choose |A|}}\,\phi(A)-\sum_{\textstyle{A\subseteq [n]\atop
|A|=n-k}}\frac{1}{{n\choose |A|}}\,\phi(A)\, .
\end{equation}
Just as for the probability $\Pr(T=X_{j})$, in the i.i.d.\ case the probability $\Pr(T=X_{k:n})$ does not depend on the c.d.f.\ $F$. Thus
$\mathbf{s}$ can be regarded as the \emph{structural signature}.

\begin{example}
For a system made up of three serially connected components with i.i.d.\ lifetimes, we have $\mathbf{s}=(1,0,0)$ and $\mathbf{b}=(1/3,1/3,1/3)$.
\end{example}

Iyer \cite{Iye92} extended the Barlow-Proschan index to the general dependent case where the c.d.f.\ $F$ is
absolutely continuous. In this setting the index $\mathbf{I}_{\mathrm{BP}}$ may depend not only on the structure function $\phi$ but
also on the c.d.f.\ $F$. Specifically, starting from the multilinear form of $\phi$,
$$
\phi(\mathbf{x}) ~=~ \sum_{A\subseteq [n]}m_{\phi}(A)\,\prod_{i\in A}x_i\, ,
$$
where $m_{\phi}\colon 2^{[n]}\to\R$ is the M\"obius transform of $\phi$, defined by
$$%\begin{equation}\label{eq:23847w}
m_{\phi}(A) ~=~ \sum_{B\subseteq A}(-1)^{|A|-|B|}\,\phi(B)\, ,
$$%\end{equation}
Iyer obtained the integral formula
\begin{equation}\label{eq:8sfd76fd}
I_{\mathrm{BP}}^{(j)} ~=~ \sum_{A\subseteq [n]\setminus\{j\}}m_{\phi}(A\cup\{j\})\,\int_0^{\infty}\frac{d}{d t_j}\Pr\Big(X_j\leqslant t_j~\mbox{~and~}~t<\min_{i\in A}X_i\Big)\Big|_{t_j=t}\, dt.
\end{equation}

The concept of signature was also extended to the general case of dependent lifetimes; see \cite{MarMat11} (see also \cite{NavSamBalBha08} for
an earlier work). Denoting this ``extended'' signature by the $n$-tuple $\mathbf{p}=(p_1,\ldots,p_n)$, where $p_k=\Pr(T=X_{k:n})$, the authors
\cite{MarMat11} proved that, if $F$ is absolutely continuous (actually the assumption that there are no ties among the component lifetimes is sufficient), then
\begin{equation}\label{eq:asghjad678}
p_k=\sum_{|A|=n-k+1}q(A)\,\phi(A)-\sum_{|A|=n-k}q(A)\,\phi(A),
\end{equation}
where the function $q\colon 2^{[n]}\to [0,1]$, called the \emph{relative quality function} associated with $F$, is defined by
$$
q(A) ~=~ \Pr\Big(\max_{i\notin A}X_i<\min_{i\in A}X_i\Big).\footnote{Thus $q(A)$ is the probability that the best $|A|$ components are precisely
those in $A$.}
$$
Thus (\ref{eq:asghjad678}) is the non-i.i.d.\ extension of (\ref{eq:asad678}). Note also that the function $q$  has the immediate property
\begin{equation}\label{eq:s7df5}
\sum_{|A|=k}q(A) ~=~ 1,\qquad k\in \{0,\ldots,n\}.
\end{equation}

In this paper, assuming only that $F$ has no ties, we give an alternative expression for $I_{\mathrm{BP}}^{(j)}$ as a weighted arithmetic mean
over $A\subseteq [n]\setminus\{j\}$ of $\Delta_j\phi(A)$ and whose coefficients depend only on $F$ (Theorem~\ref{thm:s7df5dssfd}), thus
providing the analog of (\ref{eq:asghjad678}) for the Barlow-Proschan index. This enables us to retrieve and interpret Iyer's formula (\ref{eq:8sfd76fd}) in this more general setting of distributions having no ties (Corollary~\ref{cor19092011}).
We give necessary and sufficient conditions on $F$ for $I_{\mathrm{BP}}^{(j)}$ to always reduce to (\ref{eq:asd465d}) regardless of the structure function considered (Proposition~\ref{prop:8sadf76}). We also
provide explicit expressions for the coefficient of $\Delta_j\phi(A)$ in the general continuous and independent continuous cases
(Propositions~\ref{prop:a78sd6} and \ref{lemma:6d5f67}) and examine the special case of independent Weibull lifetimes, which includes the exponential model
(Corollary~\ref{cor:vx6cx6c5}). Finally, we show how the Barlow-Proschan index can be used to measure a symmetry degree of any system
(Section~\ref{sec:a6d5a}).

%---------------------------------------------------------------------------------------------- Section 2
\section{Explicit expressions}
\label{sec:8asdf6}
Throughout we assume that the joint c.d.f.\ $F$ of the lifetimes has \emph{no ties}, i.e., we have $\Pr(X_i = X_j) = 0$ for every
$i\neq j$. For every $j\in [n]$, we define the function $q_j\colon 2^{[n]\setminus\{j\}}\to [0,1]$ as
\begin{equation}\label{eq:sds5556}
q_j(A) ~=~ \Pr\Big(\max_{i\notin A\cup\{j\}}X_i<X_j<\min_{i\in A}X_i\Big)\, .
\end{equation}
For instance, when $n=4$ we have
\begin{eqnarray*}
q_2(\{1,3\}) &=& \Pr(X_4<X_2<\min\{X_1,X_3\})\\
&=& \Pr(X_4<X_2<X_1<X_3)+\Pr(X_4<X_2<X_3<X_1)\, .
\end{eqnarray*}
From this example we immediately see that (\ref{eq:sds5556}) can be rewritten as
\begin{equation}\label{eq:sds555}
q_j(A) ~=~ \sum_{\textstyle{\sigma\in\mathfrak{S}_n{\,}:{\,}\{\sigma(n-|A|+1),\ldots,\sigma(n)\}=A\atop \sigma(n-|A|)=j }}\Pr(X_{\sigma(1)}<\cdots <X_{\sigma(n)})\, ,
\end{equation}
where $\mathfrak{S}_n$ denotes the set of permutations on $[n]$.

Thus defined, $q_j(A)$ is the probability that the components that are better than component $j$ are precisely those in $A$.\footnote{By definition the functions $q_j$ ($j\in [n]$) clearly depend only on the distribution function $F$ (and not on the structure function $\phi$).} It then follows
immediately that
\begin{equation}\label{eq:87sdaf6}
\sum_{A\subseteq [n]\setminus\{j\}}q_j(A) ~=~ 1,\qquad j\in [n].
\end{equation}
We also observe that
\begin{eqnarray}
q(A) &=& \sum_{j\notin A}q_j(A),\qquad A\neq [n],\label{eq:7sdf5}\\
q(A) &=& \sum_{j\in A}q_j(A\setminus\{j\}),\qquad A\neq\varnothing.\label{eq:7sdf5a}
\end{eqnarray}
Moreover, $q_j(\varnothing)=q(\{j\})$ is the probability that component $j$ is the best component, while
$q_j([n]\setminus\{j\})=q([n]\setminus\{j\})$ is the probability that component $j$ is the worst component.

\begin{proposition}\label{prop:s7dff5dssfd}
If the variables $X_1,\ldots,X_n$ are exchangeable, then
\begin{equation}\label{eq:qkaa}
q_j(A) ~=~ \frac{1}{(n-|A|){n\choose |A|}} ~=~ \frac{1}{n{n-1\choose |A|}}
\end{equation}
for every $j\in
[n]$ and every $A\subseteq [n]\setminus\{j\}$.
\end{proposition}

\begin{proof}
Since the variables $X_1,\ldots,X_n$ are exchangeable, by (\ref{eq:sds555}) we immediately obtain
$$
q_j(A) ~=~ \frac{1}{{n\choose n-|A|-1{\,},{\,}1{\,},{\,}|A|}}\, ,
$$
where the denominator is a multinomial coefficient.
\end{proof}

We now give an expression for $I_{\mathrm{BP}}^{(j)}$ in terms of the functions $q_j$ and $\Delta_j\phi$. This expression, given in
(\ref{eq:78s6ss7a}) below, clearly extends (\ref{eq:asd465d}) just as formula (\ref{eq:asghjad678}) extends (\ref{eq:asad678}).

\begin{theorem}\label{thm:s7df5dssfd}
For every $j\in [n]$, we have
\begin{equation}\label{eq:78s6ss7a}
I_{\mathrm{BP}}^{(j)} ~=~ \sum_{A\subseteq [n]\setminus\{j\}} q_j(A)\,\Delta_j\phi(A) ~=~ \sum_{A\subseteq [n]}(-1)^{|\{j\}\setminus A|}\,
q_j(A\setminus\{j\})\, \phi(A).
\end{equation}
\end{theorem}

\begin{proof}
For every fixed $\sigma\in\mathfrak{S}_n$ we must have
\begin{equation}\label{eq:8s6ffs}
\Pr(T=X_j\mid X_{\sigma(1)}<\cdots <X_{\sigma(n)}) ~=~  \phi\big(\{\sigma(i),\ldots,\sigma(n)\}\big)-\phi\big(\{\sigma(i+1),\ldots,\sigma(n)\}\big),
\end{equation}
where $i=\sigma^{-1}(j)$. Indeed, the left-hand expression of (\ref{eq:8s6ffs}) takes its values in $\{0,1\}$ and is exactly $1$ if and only if
$\{\sigma(1),\ldots,\sigma(i-1)\}$ is not a cut set and $\{\sigma(1),\ldots,\sigma(i)\}$ is a cut set.\footnote{Recall that a subset $K\subseteq
[n]$ of components is a \emph{cut set} for the function $\phi$ if $\phi([n]\setminus K)=0$.} That is,
$$
\phi\big(\{\sigma(i),\ldots,\sigma(n)\}\big) ~=~ 1\quad\mbox{and}\quad\phi\big(\{\sigma(i+1),\ldots,\sigma(n)\}\big) ~=~ 0.
$$
Now, by combining (\ref{eq:8s6ffs}) with the law of total probability, we get
$$
I_{\mathrm{BP}}^{(j)}  ~=~
\sum_{\sigma\in\mathfrak{S}_n}\Big(\phi\big(\{\sigma(\sigma^{-1}(j)),\ldots,\sigma(n)\}\big)-\phi\big(\{\sigma(\sigma^{-1}(j)+1),\ldots,\sigma(n)\}\big)\Big)\,\Pr(X_{\sigma(1)}<\cdots
<X_{\sigma(n)}).
$$
Grouping the terms for which $\{\sigma(\sigma^{-1}(j)+1),\ldots,\sigma(n)\}$ is a fixed set $A$ and then summing over $A$, we obtain
$$
I_{\mathrm{BP}}^{(j)}  ~=~ \sum_{A\subseteq [n]\setminus\{j\}} \big(\phi(A\cup\{j\})-\phi(A)\big) \sum_{\textstyle{\sigma\in\mathfrak{S}_n{\,}:{\,}\{\sigma(n-|A|+1),\ldots,\sigma(n)\}=A\atop \sigma(n-|A|)=j }}\Pr(X_{\sigma(1)}<\cdots <X_{\sigma(n)}).
$$
The result then follows from (\ref{eq:sds555}). The second expression in (\ref{eq:78s6ss7a}) follows immediately from the first one.
\end{proof}

\begin{example}
Assume that $\phi$ defines a $k$-out-of-$n$ structure, that is, $\phi(\bfx)=x_{k:n}$, where $x_{k:n}$ is the $k$th order statistic of the
variables $x_1,\ldots,x_n$. In this case we have $\phi(A)=1$ if and only if $|A|\geqslant n-k+1$ and hence, for every $j\in [n]$, we have
$\Delta_{j}\phi(A)=1$ if and only if $|A|=n-k$. By Theorem~\ref{thm:s7df5dssfd},
$$
I_{\mathrm{BP}}^{(j)} ~=~ \sum_{\textstyle{A\subseteq [n]\setminus\{j\}\atop |A|=n-k}} q_{j}(A)
$$
is the probability that component $j$ has the $k$th smallest lifetime. This result was expected since
$I_{\mathrm{BP}}^{(j)}=\Pr(X_{j}=T)=\Pr(X_{j}=X_{k:n})$ by definition.
\end{example}

\begin{example}
Consider a $5$-component system whose structure function $\phi\colon\{0,1\}^5\to\{0,1\}$ is defined by
$$
\phi(x_1,\ldots,x_5) ~=~ x_1\, x_4\amalg x_2\, x_5\amalg x_1\, x_3\, x_5\amalg x_2\, x_3\, x_4
$$
(see Figure~\ref{fig:bs}), where $\amalg$ is the binary coproduct (Boolean disjunction) operation defined by $x\amalg y=1-(1-x)(1-y)$. In this case we have $\Delta_3\phi(A)=1$ if and only if $A=\{1,5\}$ or $A=\{2,4\}$. By Theorem~\ref{thm:s7df5dssfd}, we then have
\begin{eqnarray*}
I_{\mathrm{BP}}^{(3)} &=& q_3(\{1,5\})+q_3(\{2,4\})\\
&=& \Pr(\max\{X_2,X_4\}<X_3<\min\{X_1,X_5\})+\Pr(\max\{X_1,X_5\}<X_3<\min\{X_2,X_4\}){\,}.
\end{eqnarray*}
Evidently this value depends on the c.d.f.\ of the component lifetimes and reduces to $1/15$ in the exchangeable case (see Proposition~\ref{prop:s7dff5dssfd}).
\end{example}

\setlength{\unitlength}{4ex}% 0.7cm 4.5ex
\begin{figure}[htbp]\centering
\begin{picture}(11,4)
\put(3,0.5){\framebox(1,1){$2$}} \put(3,2.5){\framebox(1,1){$1$}} \put(5,1.5){\framebox(1,1){$3$}} \put(7,0.5){\framebox(1,1){$5$}}
\put(7,2.5){\framebox(1,1){$4$}}%
\put(0,2){\line(1,0){1.5}}\put(1.5,2){\line(2,-1){1.5}}\put(5.5,0){\line(-2,1){1.5}}\put(1.5,2){\line(2,1){1.5}}\put(5.5,4){\line(-2,-1){1.5}}%
\put(0,2){\circle*{0.15}}%
\put(9.5,2){\line(1,0){1.5}}\put(5.5,0){\line(2,1){1.5}}\put(9.5,2){\line(-2,-1){1.5}}\put(5.5,4){\line(2,-1){1.5}}\put(9.5,2){\line(-2,1){1.5}}%
\put(11,2){\circle*{0.15}}%
\put(5.5,0){\line(0,1){1.5}}\put(5.5,4){\line(0,-1){1.5}}
\end{picture}
\caption{Bridge structure} \label{fig:bs}
\end{figure}
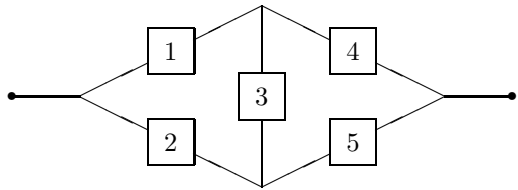

Formula (\ref{eq:78s6ss7a}) provides an explicit expression for the Barlow-Proschan index under the sole assumption that $F$ has no ties, which leads to easy interpretations and computations and reveals an interesting analogy with the concept of signature. Combining this formula with (\ref{eq:87sdaf6}) also shows that $\mathbf{I}_{\mathrm{BP}}$ is a \emph{probabilistic value}, as
defined in game theory by Weber \cite{Web88}. Moreover, $\mathbf{I}_{\mathrm{BP}}$ is \emph{efficient} in the sense that $\sum_{j=1}^nI_{\mathrm{BP}}^{(j)}=1$.

We now use (\ref{eq:78s6ss7a}) to derive Iyer's formula (\ref{eq:8sfd76fd}) in an interpretable form and without the absolute continuity assumption. For every $j\in [n]$, we define the function $r_j\colon 2^{[n]\setminus\{j\}}\to [0,1]$ as
\[
r_j(A) ~=~ \Pr\Big(X_j<\min_{i\in A}X_i\Big).
\]
That is, $r_j(A)$ is the probability that the components in $A$ are better than component $j$. We then have
\begin{equation}\label{eq:6erwe}
r_j(A)~=~\sum_{\textstyle{B\subseteq [n]\setminus\{j\}\atop B\supseteq A}}q_j(B)\, ,
\end{equation}
which can be inverted into
\begin{equation}\label{eq:6erwe2}
q_j(A)~=~\sum_{\textstyle{B\subseteq [n]\setminus\{j\}\atop B\supseteq A}}(-1)^{|B|-|A|}\, r_j(B).
\end{equation}

\begin{corollary}\label{cor19092011}
For every $j\in [n]$, we have
$$
I_{\mathrm{BP}}^{(j)} ~=~ \sum_{A\subseteq [n]\setminus\{j\}} r_j(A)\, m_{\phi}(A\cup\{j\}).
$$
\end{corollary}

\begin{proof}
Using the inverse M\"obius transform $\phi(A)=\sum_{B\subseteq A}m_{\phi}(B)$ in the right-hand side of (\ref{eq:78s6ss7a}) and then permuting the resulting sums, we obtain
\[
I_{\mathrm{BP}}^{(j)} ~=~ \sum_{B\subseteq [n]}m_{\phi}(B)\,\sum_{A\supseteq B}(-1)^{|\{j\}\setminus A|}\,
q_j(A\setminus\{j\}).
\]
Noticing that the inner sum vanishes whenever $B\not\ni j$ and setting $A'=A\setminus \{j\}$ and $B'=B\setminus\{j\}$, we obtain
\[
I_{\mathrm{BP}}^{(j)} ~=~ \sum_{B'\subseteq [n]\setminus\{j\}}m_{\phi}(B'\cup\{j\})\,\sum_{\textstyle{A'\subseteq [n]\setminus\{j\}\atop A'\supseteq B'}}q_j(A').
\]
We then conclude by (\ref{eq:6erwe}).
\end{proof}

Just as for the Barlow-Proschan index, the signature $\mathbf{p}$ also has an interesting expression in terms of the M\"obius transform of $\phi$.

\begin{proposition}\label{prop:fs87f}
For every $k\in [n]$, we have
$$%\begin{equation}\label{eq:cor1909201d1}
p_k ~=~ \sum_{A\subseteq [n]} m_{\phi}(A)\,\Pr\Big(X_{k:n}=\min_{i\in A}X_i\Big),
$$%\end{equation}
where $\Pr(X_{k:n}=\min_{i\in A}X_i)$ is the probability that the $k$th failure is that of the worst component in $A$.
\end{proposition}

\begin{proof}
Using the inverse M\"obius transform, for every $k\in [n]$ we have
$$
\sum_{|B|=k}q(B)\,\phi(B) ~=~ \sum_{|B|=k}q(B)\,\sum_{A\subseteq B}m_{\phi}(A) ~=~ \sum_{A\subseteq [n]}m_{\phi}(A)\,\sum_{\textstyle{B\supseteq A\atop |B|=k}} q(B)\, ,
$$
where the inner sum is the probability $\Pr(X_{n-k:n}<\min_{i\in A}X_i)$ that the components in $A$ be among the best $k$ components. We then conclude by (\ref{eq:asghjad678}).
\end{proof}

\begin{remark}
The proof of Proposition~\ref{prop:fs87f} shows that
$$
\Pr\Big(X_{k:n}=\min_{i\in A}X_i\Big) ~=~ \sum_{\textstyle{B\supseteq A\atop |B|=n-k+1}} q(B)-\sum_{\textstyle{B\supseteq A\atop |B|=n-k}} q(B)
$$
is exactly the $k$th coordinate $p_k$ of the signature of the semicoherent system obtained from the current system by transforming the structure function into $\phi(\bfx)=\prod_{i\in A}x_i$.\footnote{This fact also follows immediately from the identity $p_k=\Pr(T=X_{k:n})$ since this modified system has lifetime $T=\min_{i\in A}X_i\overline{\overline{\overline{\overline{}}}}$.} It also shows that the \emph{tail signature} \cite{MarMat11}, defined by $\Pr(T>X_{k:n})=\sum_{i=k+1}^np_i$, has the M\"obius representation
$$
\Pr(T>X_{k:n})~=~\sum_{|A|=n-k}q(A)\,\phi(A) ~=~ \sum_{A\subseteq [n]} m_{\phi}(A)\,\Pr\Big(X_{k:n}<\min_{i\in A}X_i\Big).
$$
\end{remark}

The next corollary, which follows immediately from Proposition~\ref{prop:s7dff5dssfd} and Theorem~\ref{thm:s7df5dssfd}, gives a sufficient
condition on $F$ for the equality $\mathbf{I}_{\mathrm{BP}}=\mathbf{b}$ to hold regardless of the structure function considered.

\begin{corollary}\label{cor:ad7s6}
If the variables $X_1,\ldots,X_n$ are exchangeable, then $\mathbf{I}_{\mathrm{BP}}=\mathbf{b}$.
\end{corollary}

We now give necessary and sufficient conditions on $F$ for the equality $\mathbf{I}_{\mathrm{BP}}=\mathbf{b}$ to hold for every structure function.

A system is said to be \emph{coherent} if it is semicoherent and its structure function $\phi$ has only \emph{essential} variables, i.e., for
every $j\in [n]$, there exists $\bfx\in\{0,1\}^n$ such that $\phi(\bfx)|_{x_j=0}\neq\phi(\bfx)|_{x_j=1}$.\footnote{In other words, every
component of the system is relevant.} Let $\Phi_n$ (resp.\ $\Phi'_n$) denote the family of $n$-variable structure functions corresponding to coherent (resp.\ semicoherent) systems.

\begin{proposition}\label{prop:8sadf76}
The equality $\mathbf{I}_{\mathrm{BP}}=\mathbf{b}$ holds for every $\phi\in\Phi_n$ (or equivalently, for every $\phi\in\Phi'_n$) if and only if (\ref{eq:qkaa}) holds for every $j\in
[n]$ and every $A\subseteq [n]\setminus\{j\}$.
\end{proposition}

\begin{proof}
We can assume that $n\geqslant 3$ (the cases $n=1$ and $n=2$ can be checked easily). Using (\ref{eq:78s6ss7a}) and (\ref{eq:asd465d}), we see
that the identity $\mathbf{I}_{\mathrm{BP}}=\mathbf{b}$ can be written as
$$
\sum_{A\subseteq [n]}(-1)^{|\{j\}\setminus A|}\, q_j(A\setminus\{j\})\, \phi(A) ~=~ \sum_{A\subseteq [n]}(-1)^{|\{j\}\setminus A|}\,
\frac{1}{n{n-1\choose |A\setminus\{j\}|}}\, \phi(A).
$$
It was shown in \cite{MarMatWal} that for any function $\lambda\colon 2^{[n]}\to\R$ we have
$$
\sum_{A\subseteq [n]}\lambda(A)\,\phi(A) ~=~ 0,\qquad\mbox{for every $\phi\in\Phi_n$ (or every $\phi\in\Phi'_n$)}
$$
if and only if $\lambda(A)=0$ for all $A\neq\varnothing$. Therefore we have $\mathbf{I}_{\mathrm{BP}}=\mathbf{b}$ for every $\phi\in\Phi_n$ (or every $\phi\in\Phi'_n$) if and only if
$q_j(A\setminus\{j\})=1/(n{n-1\choose |A\setminus\{j\}|})$ for every $j\in [n]$ and every $A\neq\varnothing$. This completes the proof.
\end{proof}

We observe that $q_j(A)$ has the form (\ref{eq:qkaa}) for every $j\in [n]$ and every $A\subseteq [n]\setminus\{j\}$ if and only if
the map $(j,A)\mapsto q_j(A)$ is symmetric in the sense that
$$
q_{\sigma(j)}(\sigma(A)) ~=~ q_j(A)
$$
for every $j\in [n]$, every $A\subseteq [n]\setminus\{j\}$, and every permutation $\sigma$ on $[n]$. Indeed, by (\ref{eq:s7df5}) and
(\ref{eq:7sdf5}), for any $k\in [n-1]$ we have
$$
\sum_{|A|=k}\,\sum_{j\notin A}q_j(A) ~=~ \sum_{|A|=k}q(A) ~=~ 1\, .
$$
The identity (\ref{eq:qkaa}) then follows from the symmetry of the map $(j,A)\mapsto q_j(A)$.\footnote{We note that this result provides an
alternative proof of Proposition~\ref{prop:s7dff5dssfd}.}

The ``signature'' version of Proposition~\ref{prop:8sadf76} can be stated as follows (see \cite{MarMatWal}). For $n\geqslant 3$ (resp.\
$n\geqslant 2$), the equality $\mathbf{p}=\mathbf{s}$ holds for every $\phi\in\Phi_n$ (resp.\ every $\phi\in\Phi'_n$) if and only if
$q(A)=1/{n\choose |A|}$ for every $A\subseteq [n]$ (i.e., $q$ is a symmetric function).

\begin{remark}\label{rem:7sfds}
It can immediately be seen that the function $q$ is symmetric (i.e., $q(A)=1/{n\choose |A|}$ for every $A\subseteq [n]$) as soon as condition (\ref{eq:qkaa}) holds for every $j\in [n]$ and every $A\subseteq [n]\setminus\{j\}$. Let us now show that the converse statement does not hold in general. Suppose $n=3$ and set $p_{ijk}=\Pr(X_i<X_j<X_k)$. Since $F$ has no ties, we must have \begin{equation}\label{eq:ergh61}
p_{123}+p_{132}+p_{213}+p_{231}+p_{312}+p_{321} ~=~ 1\, .
\end{equation}
It is clear that the function $q$ is symmetric if and only if
\begin{eqnarray}
&& p_{231}+p_{321} ~=~ p_{132}+p_{312} ~=~ p_{123}+p_{213} ~=~ 1/3\, ,\label{eq:ergh62}\\
&& p_{312}+p_{321} ~=~ p_{213}+p_{231} ~=~ p_{123}+p_{132} ~=~ 1/3\, .\label{eq:ergh63}
\end{eqnarray}
Combining Eq.~(\ref{eq:ergh61})--(\ref{eq:ergh63}) with the fact that $0\leqslant p_{ijk}\leqslant 1$, we see that $q$ is symmetric if and only if there exists $\lambda\in [0,1]$ such that
$$
(p_{123},p_{132},p_{213},p_{231},p_{312},p_{321}) ~=~ (\lambda,1-\lambda,1-\lambda,\lambda,\lambda,1-\lambda)/3\, .
$$
On the other hand, we have $q_j(A)=1/(3{2\choose |A|})$ for every $j\in [3]$ and every $A\subseteq [3]\setminus\{j\}$ if and only if $p_{ijk}=1/6$ for every permutation $(i,j,k)$ of $(1,2,3)$.
\end{remark}

%---------------------------------------------------------------------------------------------- Section 3
\section{Continuous and independent continuous lifetimes}

We now assume that the component lifetimes are absolutely continuous with p.d.f.\ $f$. This assumption enables us to derive explicit integral formulas for $q_j(A)$ and $r_j(A)$.

\begin{proposition}\label{prop:a78sd6}
For absolutely continuous lifetimes, we have
\begin{equation}\label{eq:7as6fd}
q_j(A) ~=~ \int_0^{\infty}\int_{\left]0,t_j\right[^{[n]\setminus(A\cup\{j\})}}\int_{\left]t_j,\infty\right[^A}f(\mathbf{t})\, d\mathbf{t}_A\, d\mathbf{t}_{[n]\setminus(A\cup\{j\})}\, dt_j
\end{equation}
and
\begin{equation}\label{eq:7as6fd2}
r_j(A) ~=~ \int_0^{\infty}\int_{\left]0,\infty\right[^{[n]\setminus(A\cup\{j\})}}\int_{\left]t_j,\infty\right[^A}f(\mathbf{t})\, d\mathbf{t}_A\, d\mathbf{t}_{[n]\setminus(A\cup\{j\})}\, dt_j
\end{equation}
for every $j\in [n]$ and every $A\subseteq [n]\setminus\{j\}$.
\end{proposition}

\begin{proof}
By definition we have $q_j(A)=\Pr(E)$, where $E$ is the event
$$
\max_{i\notin A\cup\{j\}}X_i<X_j<\min_{i\in A}X_i\, ,
$$
which can be described by the set
$$
\big\{(t_1,\ldots,t_n)\in\left]0,\infty\right[^n : t_i<t_j~\forall i\notin A\cup\{j\},\, t_j<t_i~\forall i\in A\big\}.
$$
Formula (\ref{eq:7as6fd}) then follows by integrating the p.d.f.\ over this event. Formula (\ref{eq:7as6fd2}) can be established similarly.
\end{proof}

We now consider the special case of independent and absolutely continuous lifetimes $X_1,\ldots,X_n$, each $X_i$ having p.d.f.\ $f_i$ and
c.d.f.\ $F_i$, with $F_i(0)=0$. The following immediate result shows how formulas (\ref{eq:7as6fd}) and (\ref{eq:7as6fd2}) can be simplified. Similar formulas for $q(A)$ can be found in \cite{MarMat11}.

\begin{proposition}\label{lemma:6d5f67}
For independent and absolutely continuous lifetimes, we have
\begin{equation}\label{eq:sd5672}
q_j(A) ~=~ \int_0^{\infty}f_j(t)\,\prod_{i\in A}\overline{F}_i(t)\prod_{i\notin A\cup\{j\}}F_i(t)\, dt
\end{equation}
and
\begin{equation}\label{eq:5sdfs}
r_j(A) ~=~ \int_0^{\infty}f_j(t)\,\prod_{i\in A}\overline{F}_i(t)\, dt
\end{equation}
for every $j\in [n]$ and every $A\subseteq [n]\setminus\{j\}$, where $\overline{F}_i(t)=1-F_i(t)$.
\end{proposition}

Using (\ref{eq:5sdfs}) and then (\ref{eq:6erwe2}) leads immediately to the following corollary.

\begin{corollary}\label{cor:vx6cx6c5}
For independent Weibull lifetimes, with $F_i(t)=1-e^{-(\lambda_i t)^{\alpha}}$, we have
\begin{equation}\label{eq:ycx7v6a}
r_j(A) ~=~ \frac{\lambda_{\alpha}(\{j\})}{\lambda_{\alpha}(A\cup\{j\})}
\end{equation}
and
\begin{equation}\label{eq:ycx7v6}
q_j(A)~=~\sum_{\textstyle{B\subseteq [n]\setminus\{j\}\atop B\supseteq A}}(-1)^{|B|-|A|}\,\frac{\lambda_{\alpha}(\{j\})}{\lambda_{\alpha}(B\cup\{j\})}
\end{equation}
for every $j\in [n]$ and every $A\subseteq [n]\setminus\{j\}$, where $\lambda_{\alpha}(A)=\sum_{i\in A}\lambda_i^{\alpha}$.
\end{corollary}

We observe that, under the assumptions of Corollary~\ref{cor:vx6cx6c5}, by (\ref{eq:ycx7v6a}) the ratio
$$
\frac{\lambda_{\alpha}(\{j\})}{\lambda_{\alpha}([n])} ~=~ r_j([n]\setminus\{j\}) ~=~ q_j([n]\setminus\{j\})
$$
is exactly the probability that $X_j$ is the shortest lifetime.

%---------------------------------------------------------------------------------------------- Section 4
\section{A symmetry index for systems}\label{sec:a6d5a}

A natural concept of symmetry for systems can be defined as follows. We say that a semicoherent system $S=(n,\phi,F)$ is \emph{symmetric} if it
has a uniform Barlow-Proschan index, i.e., $\mathbf{I}_{\mathrm{BP}}=(1/n,\ldots,1/n)$.

\begin{example}
For a system made up of $n$ serially connected components, we have $I_{\mathrm{BP}}^{(j)}=q_j([n]\setminus\{j\})=q([n]\setminus\{j\})$ for every
$j\in [n]$. The system is then symmetric if and only if $q([n]\setminus\{j\})$ is independent of $j$.
\end{example}

Since a system is rarely symmetric in the non-i.i.d.\ case, it is natural to define an index measuring a ``symmetry degree'' of the system.

Recall from probability theory that the uniformity of a probability distribution $\mathbf{w}=(w_1,\ldots,w_n)$ over $[n]$ can be measured
through the concept of \emph{normalized Shannon entropy}
$$
H(\mathbf{w}) ~=~ -\frac{1}{\ln n}\,\sum_{i=1}^nw_i\,\ln (w_i)\, ,
$$
with the convention that $0\ln 0=0$. It is well known that $H(\mathbf{w})$ is maximum $(H(\mathbf{w})=1)$ if and only if $\mathbf{w}$ is the
uniform distribution $\mathbf{w}^*=(1/n,\ldots,1/n)$ and minimum $(H(\mathbf{w})=0)$ if and only if $w_j=1$ for some $j\in [n]$ and $w_i=0$ for
all $i\neq j$ (Dirac measure). Moreover, for any probability distribution $\mathbf{w}\neq\mathbf{w}^*$, the expression
$H(\mathbf{w}_{\lambda})$, where $\mathbf{w}_{\lambda}=\mathbf{w}+\lambda\,(\mathbf{w}^*-\mathbf{w})$, strictly increases as the parameter
$\lambda$ increases from $0$ to $1$. %\footnote{In other words, $H(\mathbf{w})$ strictly increases whenever $\mathbf{w}$ moves closer to $\mathbf{w}^*$.}
Thus, the number $H(\mathbf{w})$, which lies in the interval $[0,1]$, measures a uniformity (evenness) degree of the probability distribution
$\mathbf{w}$.

On the basis of these observations we define a symmetry index as follows.

\begin{definition}
The \emph{symmetry index} for a semicoherent system $S=(n,\phi,F)$ is the number $H(\mathbf{I}_{\mathrm{BP}})$, that is, the normalized Shannon entropy of $\mathbf{I}_{\mathrm{BP}}$.
\end{definition}

It seems intuitive---at least in the exchangeable case---that the lower $H(\mathbf{I}_{\mathrm{BP}})$ (i.e., the more concentrated the distribution $\mathbf{I}_{\mathrm{BP}}$) the more the system components play an asymmetric role. Similarly, the lower $H(\mathbf{p})$ the more the components play a symmetric role. This observation is confirmed by the next two propositions, where we provide conditions under which the entropies $H(\mathbf{I}_{\mathrm{BP}})$ and $H(\mathbf{p})$ reach their extreme values.

%In the next two propositions we provide conditions under which the entropies $H(\mathbf{I}_{\mathrm{BP}})$ and $H(\mathbf{p})$ reach their extreme values. These conditions reveal that the tuples $\mathbf{I}_{\mathrm{BP}}$ and $\mathbf{p}$ in a sense may have opposite behaviors.

\begin{proposition}\label{prop:p1}
Let $S=(n,\phi,F)$ be a semicoherent system.
\begin{enumerate}
\item[$(i)$] If the functions $q_i$ $(i=1,\ldots,n)$ are strictly positive, then $H(\mathbf{I}_{\mathrm{BP}})=0$ if and only if
$\phi(\mathbf{x}) = x_j$ for some $j\in [n]$ (i.e., exactly one component of $S$ is relevant).

\item[$(ii)$] If the function $q$ is strictly positive, then $H(\mathbf{p})=0$ if and only if $\phi(\mathbf{x}) = x_{k:n}$ for some
$k\in [n]$ (i.e., $S$ is a $k$-out-of-$n$ system).\end{enumerate}
\end{proposition}

\begin{proof}
Let us prove $(i)$. Suppose first that $H(\mathbf{I}_{\mathrm{BP}})=0$. This means that $I_{\mathrm{BP}}^{(j)}=1$ for some
$j\in [n]$. Equivalently, by (\ref{eq:78s6ss7a}) we have
$$
\sum_{A\subseteq [n]\setminus\{i\}} q_i(A)\,\Delta_i\phi(A) ~=~
\begin{cases}
1, & \mbox{if $i=j$}\\
0, & \mbox{otherwise}.
\end{cases}
$$
By (\ref{eq:87sdaf6}) and due to the positivity of $q_i$, this means that $\Delta_i\phi(A)=1$ if and only if $i=j$. Equivalently,
$\phi(A)=1$ if and only if $j\in A$. Thus, $\phi(\mathbf{x}) = x_j$. The converse implication immediately follows from (\ref{eq:87sdaf6}) and (\ref{eq:78s6ss7a}).

Let us prove $(ii)$. Suppose first that $H(\mathbf{p})=0$. This means that $p_k=1$ for some $k\in [n]$. Equivalently, by
(\ref{eq:asghjad678}) we have,
$$
\sum_{|A|=i}q(A)\,\phi(A) ~=~
\begin{cases}
1, & \mbox{if $i\geqslant n-k+1$}\\
0, & \mbox{otherwise}.
\end{cases}
$$
By (\ref{eq:s7df5}) and due to the positivity of $q$, this condition means that $\phi(A)=1$ if and only if $|A|\geqslant n-k+1$. Thus,
$\phi(\mathbf{x}) = x_{k:n}$.  The converse implication immediately follows from (\ref{eq:asghjad678}) and (\ref{eq:s7df5}).
\end{proof}

\begin{proposition}\label{prop:p2}
Let $S=(n,\phi,F)$ be a semicoherent system.
\begin{enumerate}
\item[$(i)$] If the functions $q_j$ $(j=1,\ldots,n)$ have the form (\ref{eq:qkaa}) and $\phi(\mathbf{x}) = x_{k:n}$ for some
$k\in [n]$, then $H(\mathbf{I}_{\mathrm{BP}})=H(\mathbf{b})=1$.

\item[$(ii)$] If the function $q$ is symmetric and $\phi(\mathbf{x}) = x_j$ for some $j\in [n]$, then $H(\mathbf{p})=H(\mathbf{s})=1$.
\end{enumerate}
\end{proposition}

\begin{proof}
Let us prove $(i)$. Since $\phi(\mathbf{x})=x_{k:n}$ for some $k$, we have $\Delta_j\phi(A)=1$ if and only if $|A|=n-k$. It follows that
$$
I_{\mathrm{BP}}^{(j)}~=~{n-1\choose n-k}\,\frac{1}{n\,{n-1\choose n-k}}~=~\frac{1}{n}
$$
and hence $H(\mathbf{I}_{\mathrm{BP}})=1$. We also have $\mathbf{I}_{\mathrm{BP}}=\mathbf{b}$ by Proposition~\ref{prop:8sadf76}.

Let us prove $(ii)$. Since $\phi(\mathbf{x})=x_j$ for some $j$, we have $\phi(A)=1$ if and only if $j\in A$. It follows that
$$
p_k ~=~ \sum_{|A|=n-k+1}\frac{1}{{n\choose |A|}}\,\phi(A)-\sum_{|A|=n-k}\frac{1}{{n\choose |A|}}\,\phi(A)~=~\frac{{n-1\choose n-k}}{{n\choose
n-k+1}}-\frac{{n-1\choose n-k-1}}{{n\choose n-k}}~=~\frac{1}{n}
$$
for $k<n$ (and also for $k=n$) and hence $H(\mathbf{p})=1$. We also have $\mathbf{p}=\mathbf{s}$ (see the paragraph before Remark~\ref{rem:7sfds}).
\end{proof}

\begin{remark}
Put in other words, Proposition~\ref{prop:p1} states that
\begin{enumerate}
\item[$(i)$] Under positiveness of the functions $q_j$, the mass of $\mathbf{I}_{\mathrm{BP}}$ is concentrated on exactly one
coordinate ($\mathbf{I}_{\mathrm{BP}}$ is a Dirac measure) if and only if exactly one component of the system is relevant.

\item[$(ii)$] Under positiveness of the function $q$, the mass of $\mathbf{p}$ is concentrated on exactly one coordinate ($\mathbf{p}$ is a
Dirac measure) if and only if the system is of $k$-out-of-$n$ type for some $k$.
\end{enumerate}
Similarly, Proposition~\ref{prop:p2} states that
\begin{enumerate}
\item[$(i)$] For a $k$-out-of-$n$ system, if the functions $q_j$ $(j=1,\ldots,n)$ have the form (\ref{eq:qkaa}), then $\mathbf{I}_{\mathrm{BP}}$ is uniform.

\item[$(ii)$] For a system with only one relevant component, if the function $q$ is symmetric, then $\mathbf{p}$ is uniform.
\end{enumerate}
Even though the positiveness of the functions $q_j$ and $q$ seems restrictive, it is a rather natural assumption and actually necessary for Proposition~\ref{prop:p1} to hold.
\end{remark}

%---------------------------------------------------------------------------------------------- Acknowledgments
\section*{Acknowledgments}

This research is supported by the internal research project F1R-MTH-PUL-12RDO2 of the University of Luxembourg.

\end{document}